\documentclass[centering,11pt,reqno]{amsart}

\usepackage[utf8]{inputenc}
\usepackage[T1]{fontenc}
\usepackage{amsmath,amsthm}
\usepackage{amsfonts,amssymb}
\usepackage[mathscr]{eucal}
\usepackage{url}
\usepackage{mdframed}
\usepackage{paralist}
\usepackage[colorlinks,cite color=blue,pagebackref=true,pdftex]{hyperref}
\usepackage[margin=2.5cm]{geometry}
\usepackage{tikz}
\usetikzlibrary{calc,shapes}
\usepackage{mathtools}
\usepackage{multicol}
\usepackage{ytableau}
\usepackage{blkarray}
\usepackage[capitalize]{cleveref}
 \usepackage[foot]{amsaddr}
\usepackage{tikz-cd}
\usepackage{ytableau}
\usepackage{bbm}
\usepackage{comment}
\usepackage{color}
\usepackage{todonotes}
\usepackage{accsupp}

\newtheorem{theorem}{Theorem}[section]
\newtheorem*{theorem*}{Theorem}
\newtheorem*{cor*}{Corollary}
\newtheorem*{conj*}{Conjecture}
\newtheorem*{lemma*}{lem}
\newtheorem*{prop*}{Proposition}
\newtheorem{lem}[theorem]{Lemma}

\theoremstyle{definition}

\renewcommand{\epsilon}{\varepsilon}

\title{Monotone Diameters of Lattice Polytopes}

\author[Black]{Alexander E. Black}
\address{Department of Mathematics, Bowdoin College}
\email{a.black@bowdoin.edu}

\begin{document}

\maketitle

\begin{abstract}  
An influential 1989 result of Naddef shows that the diameters of $0/1$-polytopes are at most their dimension. This was extended shortly after by Kleinschmidt and Onn to any lattice polytope in $[0,k]^{d}$, where they showed a bound of at most $dk$. Naddef's argument easily extends to the monotone setting motivated by the simplex method, where one requires paths to increase with respect to a linear objective function. However, the Kleinschmidt-Onn argument does not. In fact, no argument in the 30 years since has managed to fill that gap. Prior to this work, it remained open whether the monotone diameter is bounded by a polynomial in $d$ and $k$ with no lower bounds suggesting any separation between the worst-case diameter and worst-case monotone diameter. Linear upper bounds hold for $k=1$ and $k=2$. However, we exhibit a sharp threshold for this question at $k = 3$ by constructing for each $d \geq 1$ a lattice polytope in $[0,3]^{6d}$ with monotone diameter at least $2^{d}-1$. In particular, the polynomial bound does not hold. Furthermore, we show that Naddef's result does not extend to the unbounded setting by exhibiting a family of unbounded polyhedra with $0/1$-vertices and diameter exponential in their dimension.
\end{abstract}

\section{Introduction}

The \textbf{combinatorial diameter} of a polyhedron is the maximal number of edges in any shortest path between a pair of vertices in the graph of the polyhedron. The simplex method solves a linear program by tracing a path in the graph of its feasible region to reach an optimal solution such that at each step the objective function increases. The run-time of the simplex method corresponds to the number of edges in the path together with the complexity of finding the path. As a consequence of this connection, combinatorial diameters of polyhedra have been studied extensively for decades. See \cite{santosprogress, BytheBook} and references therein for a sample of the work done on the topic. 

The additional restriction of objective function increase is called \textbf{monotonicity}. The \textbf{monotone diameter} of a polytope is the maximum distance to a sink across all orientations of its graph induced by a linear objective function. Both monotone diameters and combinatorial diameters yield lower bounds for the run-time of a simplex method with any pivot rule for solving linear programs over that polytope.   

 Naddef proved that the combinatorial diameter of any $0/1$-polytope, a polytope with vertices in $\{0,1\}^{d}$, is at most its dimension by a simple argument \cite{naddef}. His argument applies easily to the monotone diameter yielding the same bound. An alternative proof appears in \cite{01MonotoneHirsch}. Furthermore, there are constructive polynomial bounds \cite{Chubanov01,SimplexDiameterProof} and a simplex method guaranteed to match the bound for its non-degenerate steps \cite{simplexzeroone}. 

In 1992, Kleinschmidt and Onn extended Naddef's argument equally simply to any \textbf{lattice polytope}, a polytope with vertices in $\mathbb{Z}^{d}$  \cite{origlattpolydiam}. They show the combinatorial diameter of a lattice polytope in  $[0,k]^{d}$ is at most $dk$. There have since been improvements on the upper bounds to $\lfloor(k-1/2)d\rfloor$ for $k \geq 2$ \cite{LattPolyDiam} and $dk - \lceil2d/3 \rceil$  for $k \geq 3$ \cite{implattdiam} with impressive work finding precise bounds in the case of lattice zonotopes that revealed connections to number theory \cite{PrimZono, LattZonoDiam, PrimitivePoint}. Furthermore, these bounds have been made constructive \cite{LatticeShadow, ShortSimpPaths, ChubanovLatt}. However, none of these works in the thirty years since the foundational work of Kleinschmidt and Onn managed to guarantee a polynomial length monotone path as the simplex method requires.  

Conjecture 1.8 of \cite{LatticeShadow} claims that monotone diameters should be bounded by a polynomial in $d$ and $k$, and several results support this conjecture. As noted, Naddef's argument covers $k = 1$ and applies to the monotone setting. The case of $k=2$ was resolved in \cite{LatticeShadow}, where a bound of $3d$ was shown. Furthermore, Kitahara and Mizuno showed in \cite{kitaharamizuno2011primal} that Dantzig's pivot rule for the simplex method takes polynomially many non-degenerate pivot steps in terms of $d$ and $k$ for any lattice polytope in equality form, i.e., of the form $\{\mathbf{x}: A \mathbf{x} = \mathbf{b}, \mathbf{x} \geq \mathbf{0}\}$ implying a polynomial monotone diameter bound in that case. The strict monotone diameter, the maximum distance from a source to a sink across all orientations induced by a linear objective function, is at most $2dk$ by Corollary 1.2.1 of  \cite{AlexThesis}. That bound may also be made algorithmic via an application of the shadow vertex pivot rule for the simplex method. Nonetheless, we show that conjecture is false even for fixed $k = 3$.

\begin{theorem} \label{thm:main}
For each $d \geq 1$, there exists a lattice polytope in $[0,3]^{6d}$ with monotone diameter at least $2^{d}-1$.
\end{theorem}

Our construction method has a further consequence. While Naddef's bound applies to polytopes with $0/1$-vertices, the argument breaks for unbounded polyhedra with $0/1$-vertices. These polyhedra appear in integer programming and combinatorial optimization as feasible regions for covering linear programs of ideal clutters. See \cite{PackingBook, AhmadThesis} for comprehensive surveys of the topic. However, to our knowledge, there are no nontrivial upper or lower bounds known on the diameters of unbounded polyhedra with $0/1$-vertices. By extending the argument from the monotone case, we show the following:

\begin{theorem}
\label{thm:main2}
For each $d \geq 2$, there exists an unbounded $0/1$-polyhedron in $\mathbb{R}^{6d}$ with combinatorial diameter at least $2^{d}-1$ and rays of its recession cone with generators in $\{-1,0,1\}^{6d}$.
\end{theorem}

Much like how Todd's counterexample to the monotone Hirsch conjecture \cite{ToddExample} arises from the Klee-Walkup counterexample to the unbounded Hirsch conjecture \cite{KleeWalkup1967}, the mechanism underlying our monotone example arises from the example in Theorem \ref{thm:main2}. In particular, any family of examples for Theorem \ref{thm:main2} yields a family of examples for Theorem \ref{thm:main} at least for $d \geq 2$.

Note that such a construction is not possible with the ray generators in $\{0,1\}^{d}$ or even with purely non-negative ray generators. In that case, each ray has positive coordinate sum, so the face minimizing the sum of coordinates objective is a bounded face. Since the sum of coordinates objective can only take at most $d+1$ values on $0/1$-vertices, one reaches that face in at most $d$ steps and then may apply Naddef's bound yielding a combinatorial diameter bound of at most $3d$.

The core geometric idea underlying the construction for the monotone case is as follows. Start with a $0/1$-polytope $\hat{Q}$ that has an exponentially long shadow simplex path. That is, there exists a pair of objectives $\mathbf{c}$ and $\mathbf{w}$ such that the set of unique maximizers of any objective in $\{\lambda \mathbf{w} + (1-\lambda) \mathbf{c}: \lambda \in [0,1]\}$ is exponential. In geometric language, this is a line segment that intersects exponentially many full dimensional cones of the normal fan of the polytope on their interiors. Then build a lattice polytope $\hat{R}$ in $[-1,1]^{6d}$ with $\mathbf{0}$ as a vertex that has a very thin normal cone $N$ containing the segment. 

The Minkowski sum $P = \hat{Q} +\hat{R}$ is, up to integer translation, in $[0,3]^{6d}$, and we find that each vertex of $\hat{Q}$ in the shadow remains a vertex. However, their normal cones in $P$ are their original normal cones intersected with this normal cone $N$ of $\mathbf{0}$ in $\hat{R}$. We choose $N$ to be so thin that the only edges with normal cones that intersect its interior are between consecutive vertices of the shadow simplex path. Consider the directed graph starting at the vertex optimizing $\mathbf{w}$ for maximizing $\mathbf{c}$. Since $\mathbf{c}$ is contained in $N$, every facet normal of $N$ pointing away from $N$ must be decreasing with respect to $\mathbf{c}$. It follows that any monotone path starting at a vertex with normal cone contained in $N$ must consist of vertices with normal cones contained in $N$ and edges with normal cones intersecting the interior of $N$. We chose $N$ to be so thin that the only edges with normal cones that intersect the interior of $N$ are between consecutive vertices of the shadow path. Thus, starting at the $\mathbf{w}$-maximizer, the only $\mathbf{c}$-increasing path is the exponentially long shadow path. Therefore, the monotone diameter of $P$ is exponential. 

For the unbounded case, instead of taking the Minkowski sum with $\hat{R}$, take the Minkowski sum of $\hat{Q}$ with the feasible cone at $\mathbf{0}$ in $\hat{R}$. Then the graph of the resulting polyhedron is precisely the shadow simplex path which has exponential diameter. 

\section{The Construction}

Throughout we denote vectors by boldface lowercase letters and numbers in $\mathbb{R}$ by lowercase letters. For a vector $\mathbf{x}$, $x_{i}$ denotes its $i$th coordinate. Any vector $\mathbf{c} \in \mathbb{R}^{d}$ corresponds to a linear objective function $\varphi(\mathbf{x}) = \mathbf{c}^{\intercal} \mathbf{x}$. However, for simplicity, in place of referring to this function, we refer to the vector. In particular, we say a face of a polytope is $\mathbf{c}$-maximal if that function is maximized on that face. A face is uniquely $\mathbf{c}$-maximal if the set of $\mathbf{c}$-maximizers is that face. 

We rely on standard terminology regarding polytopes as found in \cite{LecturesonPolytopes} especially as pertains to normal fans and Minkowski sums. To start, we prove a technical lemma regarding Minkowski sums that is the key polyhedral insight driving the proof.

\begin{lem}
\label{lem:MinkSum}
Let $Q, R \subseteq \mathbb{R}^{d}$ be polytopes,  let $\mathbf{c} \in \mathbb{R}^{d}$ be a linear objective, and let $\mathbf{v}$ be a vertex of $R$. Suppose that $\mathbf{v} \in R$ is the unique $\mathbf{c}$-maximizer. Let $S \subseteq Q$ be the set of vertices of $Q$ that are simultaneously uniquely maximal with $\mathbf{v}$ for some linear objective. Then $S +\mathbf{v} = \{\mathbf{u} + \mathbf{v}: \mathbf{u} \in S\}$ is a subset of the vertices of $Q+R$, and any $\mathbf{c}$-increasing path starting at a vertex in $S + \mathbf{v}$ in $Q+R$ stays in $S + \mathbf{v}$ and only uses edges uniquely maximized by some objective uniquely maximized at $\mathbf{v}$. 
\end{lem}

\begin{proof}
Recall that the normal fan of $Q+R$ is the common refinement of the normal fans of $Q$ and $R$. In particular, the vertices of $Q +R$ are precisely the sums of vertices in $Q$ and $R$ that are simultaneously uniquely maximal for some linear objective. By hypothesis, this is the case for every vertex in $S + \mathbf{v}$, so they are vertices of $Q+R$.

Let $N$ denote the normal cone of $\mathbf{v}$ in $R$. The normal cone of each vertex $\mathbf{u} + \mathbf{v} \in S +\mathbf{v}$ is the normal cone of $\mathbf{u}$ in $Q$ intersected with the normal cone of $\mathbf{v}$ in $R$. In particular, the normal cone of $\mathbf{u} + \mathbf{v}$ is contained in $N$. Furthermore, by definition, $S+\mathbf{v}$ consists of all vertices of $Q+R$ with a normal vector contained in the interior of $N$, and so their normal cones subdivide $N$. Thus, for a path to leave $S +\mathbf{v}$, it would have to leave $N$. Since the normal fan of $Q+R$ is the common refinement of the two normal fans, such a path would have to leave along a supporting hyperplane of $N$. In the graph of the polytope, this would mean following an edge parallel and in the same orientation as one leaving $\mathbf{v}$ in $R$. Since $\mathbf{v}$ is the unique $\mathbf{c}$-maximizer on $R$, any such edge must be $\mathbf{c}$-decreasing. Hence, no $\mathbf{c}$-increasing path can take such an edge. Therefore, any $\mathbf{c}$-increasing path on $Q+R$ starting in $S +\mathbf{v}$ must stay in $S + \mathbf{v}$. By the same reasoning, the normal of any edge used must be a boundary between two cones in the subdivision of $N$ and therefore contains a linear objective in the interior of $N$ and thus uniquely maximized at $\mathbf{v}$.
\end{proof}

 All that remains is to construct our pair of polytopes. Let $Q = \text{conv}(\{(\mathbf{r}(t), \mathbf{s}(t)) \in \{0,1\}^{3d}: 0 \leq t \leq 2^{d}-1, t\in \mathbb{Z}\})$, where $\mathbf{r}(t)$ are the $d$ digits in the binary expansion of $t$ and $\mathbf{s}(t)$ are the $2d$ digits of the binary expansion of $t^{2}$. Throughout, we let $\mathbf{v}(t) = (\mathbf{r}(t), \mathbf{s}(t))$. This is the $0/1$ extended formulation for the convex hull of integer points on a parabola. To be more explicit, let $\mathbf{y} = (1,2,2^{2},\dots,2^{d-1},0,\dots,0)$ and $\mathbf{z} = (0,\dots,0,1,2,\dots,2^{2d-1})$. Then for all $ t \in [0,2^{d}-1] \cap \mathbb{Z}$, $\pi(\mathbf{v}(t)) = (t,t^{2})$, where
\[\pi(\mathbf{x}) = (\mathbf{y}^{\intercal}\mathbf{x}, \mathbf{z}^{\intercal}\mathbf{x})= \left(\sum_{i=1}^{d} x_{i}2^{i-1}, \sum_{j=d+1}^{3d} x_{j}2^{j-d-1}\right).\]
In particular, $\pi(Q) =\text{conv}( \{(t,t^{2}): t \in [0,2^{d}-1] \cap \mathbb{Z}\})$. For example, for $d=2$, the vertices are 
\[\mathbf{v}(0) = (0,0,0,0,0,0), \mathbf{v}(1) = (1,0,1,0,0,0), \mathbf{v}(2) = (0,1,0,0,1,0), \text{ and }
    \mathbf{v}(3) = (1,1,1,0,0,1).\]

Since $x^{2}$ is a strictly convex function, $\pi(\mathbf{v}(t)) = (t,t^{2})$ is a vertex of $\pi(Q)$ for all $t \in [0,2^{d}-1]\cap \mathbb{Z}$, so $\pi(Q)$ has $2^{d}$ many vertices. This is a standard example of a $0/1$-polytope with an exponentially large shadow as one may find in Section 3.1 of \cite{01Extremal}. 

We define $R$ by 
\[R = \text{conv}(\mathbf{0} \cup \{\mathbf{v}(t+1)-2\mathbf{v}(t) + \mathbf{v}(t-1): t \in [1,2^{d}-2] \cap \mathbb{Z}\}) \subseteq \mathbb{R}^{3d}.\]
We choose $R$ in this way to ensure an objective $\mathbf{c}$ is uniquely maximized at $\mathbf{0}$ over $R$ if and only if the sequence $\mathbf{c}^{\intercal} \mathbf{v}(0), \mathbf{c}^{\intercal} \mathbf{v}(1), \dots, \mathbf{c}^{\intercal} \mathbf{v}(2^{d}-1)$ is strictly concave. This is the key property to generalize from linear combinations of $\mathbf{y}$ and $\mathbf{z}$ . In particular, any positive combination of $\mathbf{y}$ and $-\mathbf{z}$ is in the interior of the normal cone of $\mathbf{0}$ in $R$.

Taking $Q+R$ yields an exponential monotone diameter example. However, $R$ is a lattice polytope in $[-2,2]^{3d}$, so $Q+R + 2\mathbf{1}$ is in $[0,5]^{3d}$. That is enough to deal with $k = 5$ and larger, but it is insufficient to reach the sharp threshold at $k= 3$. Thus, we need to take one additional step to correct this.

We double the coordinates of $Q$ by embedding it in twice the ambient dimension. Define $\hat{Q} = \{(\mathbf{x}, \mathbf{x}): \mathbf{x} \in Q\} \subseteq [0,1]^{6d}$. Note that we can rewrite the expression used for defining the vertices of $R$ that forces strict concavity as
\[\mathbf{v}(t+1)-2\mathbf{v}(t) + \mathbf{v}(t-1) = (\mathbf{v}(t+1) - \mathbf{v}(t)) + (\mathbf{v}(t-1) - \mathbf{v}(t)).\] 
The motivation behind the doubling procedure is that $\mathbf{v}(t-1) -\mathbf{v}(t)$ and $\mathbf{v}(t+1) - \mathbf{v}(t)$ are both differences of two $\{0,1\}$ vectors and thus have coordinates in $\{-1,0,1\}$. In particular, we define
\begin{align*}
    \hat{R} = \text{conv}(\mathbf{0} &\cup \{(\mathbf{v}(t+1) -\mathbf{v}(t), \mathbf{v}(t-1) -\mathbf{v}(t)): t \in [1,2^{d}-2] \cap \mathbb{Z}\} \\ 
    &\cup \{(\mathbf{v}(t-1) -\mathbf{v}(t), \mathbf{v}(t+1) -\mathbf{v}(t)): t \in [1,2^{d}-2] \cap \mathbb{Z}\}).  \end{align*}
Then when $\mathbf{0}$ is uniquely maximal for a given objective over $\hat{R}$, the sequence of objective values on $(\mathbf{v}(s), \mathbf{v}(s))$ is still strictly concave, but additionally $\hat{R} \subseteq [ -1,1]^{6d}$. We use this insight to prove our main theorem. Notationally, we will frequently need to double the coordinates of our vectors, and we denote this by $\hat{\mathbf{x}} = (\mathbf{x},\mathbf{x})$.

\begin{proof}[Proof of Theorem \ref{thm:main}]
Consider $P=\hat{Q} + \hat{R}$. By definition $\hat{Q}$ is a $0/1$-polytope. Furthermore, each vertex in $\hat{R}$ is either $\mathbf{0}$ or a difference of $0/1$ vectors, so $\hat{R}$ is a lattice polytope contained in $[-1,1]^{6d}$. Therefore, since $\hat{Q}$ is a $0/1$-polytope, $P+\mathbf{1} \subseteq [0,3]^{6d}$ and is a lattice polytope.  

For each $\alpha, \beta > 0$, let $\mathbf{c} = \alpha\mathbf{y} -\beta \mathbf{z}$, where we recall that 
$\mathbf{y} = (1,2,\dots,2^{d-1},0,\dots,0) \in \mathbb{R}^{3d}$ and \\$\mathbf{z} = (0,\dots,0, 1,2,\dots,2^{2d-1}) \in \mathbb{R}^{3d}$. 

Then for any $t \in [0,2^{d}-1] \cap \mathbb{Z}$, by construction of $\hat{Q}$,
\[\mathbf{c}^{\intercal}\mathbf{v}(t) = \alpha t-\beta t^{2} \text{ and } \hat{\mathbf{c}}^{\intercal} \hat{\mathbf{v}}(t) = 2(\alpha t -\beta t^{2}).\]
This function is strictly concave as a function of $t$ in both cases, so 
\begin{align*}          (\mathbf{c},\mathbf{c})^{\intercal}(\mathbf{v}(t+1)-\mathbf{v}(t), \mathbf{v}(t-1)-\mathbf{v}(t)) &= (\mathbf{c},\mathbf{c})^{\intercal}(\mathbf{v}(t-1)-\mathbf{v}(t), \mathbf{v}(t+1)-\mathbf{v}(t))\\
&= \mathbf{c}^{\intercal}(\mathbf{v}(t+1)-2\mathbf{v}(t)+\mathbf{v}(t-1)) \\
&<0.
\end{align*}
Therefore, $\hat{\mathbf{c}}$ is uniquely maximized at $\mathbf{0}$ in $\hat{R}$. In particular, $\mathbf{0}$ is a vertex of $\hat{R}$. 

Furthermore, by adjusting the choice of $\alpha, \beta > 0$, we may ensure that, for any fixed choice of $t$, $\hat{\mathbf{v}}(t)$ is a $\hat{\mathbf{c}}$-maximizer via an elementary calculus argument. 

If $t = 0$, take any $0 < \alpha < \beta$. Then $\alpha s- \beta s^{2} < 0$ for all $s \geq 1$ meaning that $0$ is  maximal.  Otherwise, $t > 0$, so there exist positive $\alpha$ and $\beta$ such that $\frac{\alpha}{2\beta} = t$. It follows then that $\alpha -2\beta t = 0$. Since $2(\alpha t- \beta t^{2})$ is strictly concave, $t$ is then a global maximizer. In particular, this implies that for all $t \in [0,2^{d}-1] \cap \mathbb{Z}$, $\hat{\mathbf{v}}(t)$ and $\mathbf{0}$ are simultaneously uniquely maximal for the same objective. Thus, $\hat{\mathbf{v}}(t) + \mathbf{0} =\hat{\mathbf{v}}(t)$ is a vertex of $P$ with normal cone contained in the normal cone of $\mathbf{0}$ in $\hat{R}$ for all $t \in [0,2^{d}-1] \cap \mathbb{Z}$.

Let $(\mathbf{a}, \mathbf{b}) \in \mathbb{R}^{3d} \times \mathbb{R}^{3d}$, and suppose that $(\mathbf{a},\mathbf{b})$ is maximized uniquely at $\mathbf{0}$ over $\hat{R}$. Then, by definition of $\hat{R}$, for any $t \in [1,2^{d}-2] \cap \mathbb{Z}$, 
\begin{align*}
    (\mathbf{a}, \mathbf{b})^{\intercal} (\mathbf{v}(t+1) -\mathbf{v}(t), \mathbf{v}(t-1) -\mathbf{v}(t)) &< 0 \\
    (\mathbf{a}, \mathbf{b})^{\intercal} (\mathbf{v}(t-1) -\mathbf{v}(t), \mathbf{v}(t+1) -\mathbf{v}(t)) &<0,
\end{align*}
so by adding these strict inequalities,
\[(\mathbf{a}, \mathbf{b})^{\intercal}(\mathbf{v}(t+1) -2\mathbf{v}(t)+\mathbf{v}(t-1), \mathbf{v}(t+1) -2\mathbf{v}(t)+\mathbf{v}(t-1)) < 0.\]
Thus, $(\mathbf{a},\mathbf{b})^{\intercal} \hat{\mathbf{v}}(t)$ is also a strictly concave function of $t$.

The inequality $(\mathbf{a},\mathbf{b})^{\intercal}(\hat{\mathbf{v}}(t+1)- 2\hat{\mathbf{v}}(t) + \hat{\mathbf{v}}(t-1)) < 0$ is equivalent to $(\mathbf{a},\mathbf{b})^{\intercal}\hat{\mathbf{v}}(t+1) - (\mathbf{a},\mathbf{b})^{\intercal}\hat{\mathbf{v}}(t) < (\mathbf{a},\mathbf{b})^{\intercal} \hat{\mathbf{v}}(t) - (\mathbf{a},\mathbf{b})^{\intercal}\hat{\mathbf{v}}(t-1)$. That is consecutive differences are strictly decreasing. This implies there is either a unique maximum in the sequence or any pair of maxima must be consecutive. It follows then that if $(\mathbf{a}, \mathbf{b})$ is maximized uniquely at $\mathbf{0}$ in $\hat{R}$ and defines an edge between vertices $\hat{\mathbf{v}}(s)$ and $\hat{\mathbf{v}}(t)$, then $|s-t| = 1$.  Hence, since the normal cones of $\hat{\mathbf{v}}(s)$ and $\hat{\mathbf{v}}(t)$ are contained in the normal cone of $\mathbf{0}$ in $\hat{R}$ in $P$, there is only an edge between $\hat{\mathbf{v}}(s)$ and $\hat{\mathbf{v}}(t)$ in $P$ if $|s-t| = 1$.

To lower bound the monotone diameter, consider the graph of $P$ oriented by the objective $-\hat{\mathbf{z}}$ starting at the vertex $\hat{\mathbf{v}}(2^{d}-1)$ of $P$. Note that $-\hat{\mathbf{z}}^{\intercal}\hat{\mathbf{v}}(\mathbf{t}) = -2t^{2}$, which is strictly concave, and it is thus uniquely maximized at $\mathbf{0}$ in $\hat{R}$ as well. In $\hat{Q}$, the largest value of $-t^{2}$ is $0$, which is attained at $\mathbf{0}=\hat{\mathbf{v}}(0)$. Thus, the monotone path must start at $\hat{\mathbf{v}}(2^{d}-1)$ and end at $\hat{\mathbf{v}}(0)$. 

Since $-\hat{\mathbf{z}}$ is maximized uniquely at $\mathbf{0}$ in $\hat{R}$, by Lemma \ref{lem:MinkSum} any $-\hat{\mathbf{z}}$ monotone path starting at $\hat{\mathbf{v}}(2^{d}-1)$ must only use vertices of the form $\hat{\mathbf{v}}(t)$ for $t \in [0,2^{d}-1] \cap \mathbf{Z}$ and edges optimized by objectives maximized uniquely at $\mathbf{0}$ in $\hat{R}$. Thus, any edge must be between pairs of vertices of the form $\hat{\mathbf{v}}(t)$ with consecutive indices. Hence, any  such monotone path starts at $\hat{\mathbf{v}}(2^{d}-1)$, ends at $\hat{\mathbf{v}}(0)$, and at each step only changes the index by $1$. Therefore, such a monotone path is of length at least $2^{d}-1$ implying the desired lower bound. 
\end{proof}

In fact, while not necessary to the conclusion, the path used is unique and corresponds to the shadow simplex path as described in our geometric outline.

\subsection{The Unbounded Case}

The unbounded construction follows as a consequence of the previous result. The following lemma is a standard polyhedral geometry fact. 

\begin{lem}
\label{lem:BoundedFaces}
Let $P \subseteq \mathbb{R}^{d}$ be a polytope, and let $C$ be a pointed cone with vertex $\mathbf{0}$. Then bounded faces of $P+ C$ are precisely the bounded faces of $P$ with normal cones that intersect the interior of $N$, the normal cone of $\mathbf{0}$ in $C$. Furthermore, for any such face $F$, its normal cone is its normal cone in $P$ intersected with $N$.
\end{lem}

\begin{proof}
If $\mathbf{c}$ is increasing on any ray of $C$, then there is either no $\mathbf{c}$-maximal face on $P+C$ or the resulting face is not bounded, because one could always weakly increase the objective by adding a multiple of that ray to a maximizer. Thus, if a face of $P+C$ is bounded, it must have an objective in its normal cone that is strictly decreasing along any ray of $C$. This is equivalent to being in the interior of $N$. Hence, the normal cone of any bounded face of $P+C$ must have relative interior contained in the interior of $N$. If an objective uniquely maximized at that face $F$ in $P$ is contained in the interior of $N$, it is optimized at $\mathbf{0}$ in $C$, and so the resulting maximal face in $P+C$ is $F + \mathbf{0} = F$. Thus, any vector in the intersection of the normal cone of $F$ and $N$ is contained in $N$. Therefore, the normal cones of the bounded faces of $P+C$ are precisely the normal cones of the faces of $P+C$ that intersect $N$ on its interior and are given by the intersection with $N$.
\end{proof}

This geometric observation allows us to take the same monotone idea and apply it to the unbounded setting.

\begin{proof}[Proof of Theorem \ref{thm:main2}]
Since $d \geq 2$, the interval $[1,2^{d}-2]$ is nonempty. Thus, $\hat{R}$ has a vertex other than $\mathbf{0}$. Take $\hat{Q} + C$, where $C$ is the feasible cone of $\mathbf{0}$ in $\hat{R}$. Since $\hat{R}$ is a lattice polytope in $[-1,1]^{6d}$, the rays of that feasible cone have generators in $\{-1,0,1\}^{6d}$. The normal cone of $\mathbf{0}$ in $C$ is the normal cone of $\mathbf{0}$ in $\hat{R}$, since it has the same feasible cone. Then by Lemma \ref{lem:BoundedFaces}, the bounded faces of $\hat{Q} + C$ are precisely the faces of $\hat{Q}$ with normal cones that intersect the interior of the normal cone of $\mathbf{0}$ in $\hat{R}$. By the proof of Theorem \ref{thm:main}, the vertices are then each of the form $\hat{\mathbf{v}}(t)$ for all $t \in [0,2^{d}-1] \cap \mathbb{Z}$ and the edges are exactly those between vertices with consecutive indices. Thus, the resulting graph is a path of length $2^{d}-1$, which has diameter $2^{d}-1$ and yields the result.
\end{proof}


\section*{AI Disclosure}
The proof of Theorem \ref{thm:main} arose from discussion with ChatGPT 5.6 Sol Pro and ChatGPT 6 Pro. The final proof here is our own, and we take full responsibility for correctness. AI was only used for checking for typos otherwise.

\section*{Acknowledgments}
I have benefited from discussion with many researchers regarding this problem especially Stefan Weltge, Laura Sanit\`{a}, Jes\'{u}s De Loera, and Sean Kafer.

\bibliographystyle{amsplain}
\bibliography{bibliography.bib}

\end{document}